\documentclass[ a4paper,10pt]{amsart}
\usepackage[foot]{amsaddr}

\usepackage[T1]{fontenc}

\usepackage[utf8]{inputenc}

\usepackage{lmodern}

\usepackage[english]{babel}

\usepackage{amsmath}
\makeatletter
\def\maketag@@@#1{\hbox{\m@th\normalfont\normalsize#1}}
\makeatother
\usepackage{amssymb}
\usepackage[alphabetic]{amsrefs}
\usepackage{mathtools}

\usepackage{mathrsfs}

\usepackage{xypic}

\usepackage{tikz-cd}

\usepackage{bookmark}

\usepackage{hyperref} 
\hypersetup{
	colorlinks=true,
	linkcolor=blue,
	filecolor=magenta,      
	urlcolor=cyan,
}

\usepackage[capitalise]{cleveref}

\usepackage{enumitem}
\usepackage{graphicx}
\usepackage{tikz}
\usepackage{tikz-cd}
\usetikzlibrary{matrix,arrows,decorations.pathmorphing}
\usepackage[all,cmtip]{xy}

\newtheorem*{thm-no-num}{Theorem}
\newtheorem*{df-no-num}{Definition}
\newtheorem{thm}{Theorem} [section]

\newtheorem{maintheorem}{\bf Theorem}
\newtheorem{prop}[thm]{Proposition} 
\newtheorem{lm}[thm]{Lemma} 
\newtheorem{cor}[thm]{Corollary} 

\theoremstyle{remark}
\newtheorem{rmk}[thm]{Remark}
\newtheorem{ex}[thm]{Example}

\newcommand{\bA}{\mathbb{A}}

\newcommand{\PP}{\mathbb{P}}
\newcommand{\ZZ}{\mathbb{Z}}
\newcommand{\FF}{\mathbb{F}}

\newcommand{\cl}[1]{\mathcal{#1}}

\newcommand*{\sheafhom}{\mathcal{H}\kern -.5pt om}

\newcommand{\Mcal}{\cl{M}}

\newcommand{\Hcal}{\cl{H}}

\newcommand{\Xcal}{\cl{X}}
\newcommand{\Ycal}{\cl{Y}}

\newcommand{\bfk}{\mathbf{k}}

\renewcommand{\H}{{\rm H}}
\newcommand{\K}{{\rm K}}
\newcommand{\M}{{\rm M}}

\newcommand{\Inv}{{\rm Inv}^{\bullet}}
\newcommand{\Spec}{{\rm Spec}}

\begin{document}
	\title[Theta characteristics and cohomology classes]{Cohomology classes on moduli of curves from theta characteristics}
	\author[A. Jaramillo Puentes]{Andrés Jaramillo Puentes}
	\address{Universit\"at T\"ubingen, Fachbereich Mathematik, Auf der Mongenstelle 10, 72076 T\"ubingen, Germany}
	\email{andres.jaramillo-puentes@math.uni-tuebingen.de}
	\author[R. Pirisi]{Roberto Pirisi}
	\address{Università degli Studi di Napoli Federico II, Dipartimento di matematica e applicazioni R. Caccioppoli, Via Cintia, Monte S. Angelo I-80126, Napoli, Italy}
	\email{roberto.pirisi@unina.it}	
	\date{\today}
\begin{abstract}
Using Galois-Stiefel-Whitney classes of theta characteristics we show that over a totally real base field the moduli stack of smooth genus $g$ curves and the moduli stack of principally polarized abelian varities of dimension $g$ have nontrivial cohomological invariants and \'etale cohomology classes in degree respectively $2^{g-2}$, $2^{g-1}$ and $2^{g-1}$. We also compute the pullback from the Brauer group of $\Mcal_3$ to that of $\Hcal_3$ over a general field of characteristic different from $2$. 
\end{abstract}
\maketitle

\section*{Introduction}
Recently A. Di Lorenzo and the second author proved \cite{DilPirM3} that given any base field $\bfk$ of characteristic $c \neq 2$ and an even positive integer $\ell$ not divisible by $c$ we have
\[ {\rm Br}(\Mcal_3)_\ell = {\rm Br}(\bfk)_\ell \oplus \ZZ/2\ZZ.\]
The result is obtained using the theory of cohomological invariants to bound the possible elements in the group and then leveraging theta characteristics and Galois-Stiefel-Whitney classes to construct an explicit generator. The paper leaves three questions open:

\begin{enumerate}
    \item What is the image of ${\rm Br}(\Mcal_3)$ inside ${\rm Br}(\Hcal_3)$?
    \item Is it possible to compute the full cohomological invariants of $\Mcal_3$, at least over an algebraically closed field?
    \item Do the \'etale algebras obtained by odd and even theta characteristics induce any nontrivial invariants on $\Mcal_g$ for $g \geq 4$?
\end{enumerate}

In this paper we answer the first and third question\footnote{The second question will be addressed in an upcoming work by A. Di Lorenzo.}, and as byproduct we exhibit some nontrivial \'etale cohomology classes for the moduli stacks of smooth curves and principally polarized abelian varieties. Our results are as follows:

\begin{maintheorem}\label{thm:pullback}
The pullback from the Brauer group of $\Mcal_3$ to the Brauer group of $\Hcal_3$ is injective and the image of the class $\alpha_2(\mathcal{S}_3^-)$ coming from odd theta characteristics is the class $\alpha_2(\mathcal{W}_3)+\lbrace-1\rbrace\alpha_1(\mathcal{W}_3)$ coming from the Weierstrass divisor.
\end{maintheorem}

\begin{maintheorem}\label{thm:classes}
Let $\bfk=\mathbb{R}$ be the field of real numbers, and $g\geq 3$. Then the following classes are $\H^{\bullet}(\mathbb{R},\ZZ/2\ZZ)$-linearly independent:
\begin{enumerate}
    \item $1, \alpha_{2^{g-2}}(\mathcal{S}_g^-) \in \H^{2^{g-2}}(\Mcal^{\rm ct}_g,\ZZ/2\ZZ), \alpha_{2^{g-1}}(\mathcal{S}_g) \in \H^{2^{g-1}}(\Mcal^{\rm ct}_g,\ZZ/2\ZZ)$,
    \item $1, \alpha_{2^{g-1}}((\mathcal{X}_g)_{2}) \in \H^{2^{g-1}}(\mathcal{A}_g,\ZZ/2\ZZ)$.
\end{enumerate}
Moreover, they are independent from the image of the cycle map, and they stay independent when pulled back to $\Mcal_g$, when pulled back to $\Hcal_g$ and when they are seen as cohomological invariants.
\end{maintheorem}

The classes $\alpha_i(\mathcal{W}_3), \alpha_i(\mathcal{S}_g^-), \alpha_i(\mathcal{S}_g)$ and $\alpha_{i}((\Xcal_g)_2)$ are obtained by pulling back elements in the cohomology of ${\rm BS}_n$ as explained in Section \ref{sec: StiefelW}. To show that these classes are nontrivial and to compare them, we proceed similarly to \cite{DilPirM3}*{Sec. 3.4}. We construct an explicit test curve over some transcendental extension of the base field, where the scheme of theta characteristics can be described explicitly, and we compute the induced Galois-Stieifel-Whitney classes. While in \emph{loc. cit.} A. Di Lorenzo and the second author construct a quartic as the ramification locus of a two-sheeted covering covering of $\PP^2$, here we construct a hyperelliptic curve over a field extension of $\bfk$ of transcendence degree $g$, and we leverage the explicit description \cite{Dolg} of theta characteristics on a hyperelliptic curve over an algebraically closed field to do our computations.
We begin with the simpler case of genus three, which allows us to prove Thm. \ref{thm:pullback} by comparing the class we obtain in degree $2$ with the generators identified in \cite{DilPirBr}, before moving on to the general case, which involves more substantial computations. Note that in proving Thm. \ref{thm:classes} we proceed, in a sense, in the opposite direction from how the theorem is stated. We show that these classes are independent as cohomological invariants of $\Hcal_g$, and this implies all other statements.

\subsection*{Notation}

Every scheme and stack will be assumed to be of finite type (or a localization thereof) over a field $\bfk$ of characteristic different from $2$.

With the letter $\ell$ we will always mean a positive integer coprime to the characteristic of $\bfk$. If $A$ is an abelian group, by $A_r$ we will always mean the $r$-torsion subgroup of $A$. If $A^{\bullet}$ is a graded group we write $A^{\bullet}\left[ s \right]$ for the graded group obtained by shifting $A^{\bullet}$ up in degree by $s$.

With the notation ${\rm H}^i(\Xcal,A)$ we always mean \'etale cohomology with coefficients in $A$, or lisse-\'etale if $\Xcal$ is not a Deligne-Mumford stack. If $R$ is a $\bfk$-algebra we will often write ${\rm H}^i(R,A)$ for ${\rm H}^i(\Spec(R),A)$. In particular for a field $F/\bfk$ we have ${\rm H}^i(F,A)={\rm H}^i_{\textnormal{\'et}}(\Spec(F),A)={\rm H}_{\textnormal{Gal}}^i(F,A)$.

\section{Preliminaries}

\subsection{Cycle modules}
Given a field $F$, the Milnor's $\K$-theory ring of $F$ is defined as the graded ring 
\[\K^{\bullet}_{\rm Mil}(F)=\oplus_i(F^*)^{\otimes i}/I\]
where $I$ is the bilateral ideal given by
\[I=\left( a_1 \otimes \ldots \otimes a_r \mid r>0, \, a_1 + \ldots + a_r = 1\right).\]
We write $\lbrace a_1, \ldots , a_r\rbrace$ for the element $a_1 \otimes \ldots \otimes a_r$. Some properties of $\K_{\rm Mil}^{\bullet}(F)$ we will need are:
\begin{itemize}
    \item $\K_{\rm Mil}^{\bullet}(F)$ is graded-commutative, i.e. \[\lbrace a_1, \ldots , a_r, b_1, \ldots, b_d \rbrace = (-1)^{rd} \lbrace b_1, \ldots, b_d, a_1, \ldots , a_r \rbrace.\]
    \item $\lbrace a_1, \ldots, a_n \rbrace$ is zero if $a_1 + \ldots + a_n = 0$.
    \item $\lbrace a, a \rbrace = \lbrace -1, a \rbrace.$
\end{itemize}
Milnor's K-theory induces a functor from the category of finitely generated field extensions of $\bfk$ to the category of graded abelian groups, and it has a series of important functorial properties:
\begin{enumerate}
    \item For any finitely generated field extension $f:F \to F'$, a restriction map 
    \[f^*:\K^{\bullet}_{\rm Mil}(F) \to \K^{\bullet}_{\rm Mil}(F'), \quad f^*\lbrace a \rbrace = \lbrace f(a) \rbrace.\]
    \item For any finite extension $f:F \to F'$, a norm map 
    \[f_*:\K^{\bullet}_{\rm Mil}(F') \to \K^{\bullet}_{\rm Mil}(F), \quad f_*\lbrace a \rbrace = \lbrace {\rm N}^{F'}_F(a) \rbrace.\]
    \item For a DVR $(R,v)$ with quotient field $\bfk(R)$ and residue field $\bfk_v$ a ramification map 
    \[\partial_v:\K^{\bullet}_{\rm Mil}(\bfk(R)) \to \K^{\bullet}_{\rm Mil}(\bfk_v), \quad \partial_v(\lbrace \pi, a_1, \ldots, a_r \rbrace)=\lbrace \overline{a_1}, \ldots, \overline{a_r} \rbrace, \]
    where $\pi$ is a uniformizer for $R$, which lowers the degree by one.
\end{enumerate}
The maps above satisfy the expected functorial compatibilities and a projection formula $f_*(f^*(\lbrace a_1, \ldots a_r\rbrace))={\rm deg}(F'/F)\lbrace a_1, \ldots a_r\rbrace$. Moreover we have the following exact sequence for Milnor's K-theory of a purely transcendental extension
\[0 \to \K_{\rm Mil}^{\bullet}(F) \to \K_{\rm Mil}^{\bullet}(F(t)) \xrightarrow{\partial} \bigoplus_{x \in \bA^1_F} \K_{\rm Mil}^{\bullet}(\bfk(x)) \to 0, \]
where the map $\partial$ is given by the sum of all ramification maps coming from points of codimension one in $\bA^1_F$.

As defined by Rost \cite{Rost}, a cycle module $\M$ is, roughly, a functor from finitely generated extensions of our base field $\bfk$ to abelian groups such that for all $F$, the group $\M(F)$ is a $\K_{\rm Mil}(F)$-module. It has the same main properties of Milnor's $K$-theory, and these are compatible with the module structure, for example there are transfer maps for finite extensions and ramification maps for geometric valuations compatible with the module structure.  
The main examples of cycle modules are $\K_{\rm Mil}$ itself, its quotients $\K_{\rm Mil}/\ell \stackrel{\rm def}{=} \K_\ell$, and twisted Galois cohomology with coefficients in a torsion Galois module $\H_D \stackrel{\rm def}{=} \oplus_i \H_{\rm Gal}^{i}(-,D(i))$. By the Norm-Residue theorem, we have the equality $\K_\ell = \H_{\ZZ/\ell\ZZ}$.

\subsection{Cohomological invariants}\label{sec:cohinv}

We briefly recall the definition of cohomological invariants and list some important properties. For an in depth discussion the reader can refer to \cite{DilPirBr}*{Sec. 2}. Fix a cycle module $\M$. Given an algebraic stack $\Xcal/\bfk$, we consider \cite{DilPirBr}*{Definition 2.3} the functor 
\[{\rm Pt}_{\Xcal}: (\mathcal{F}/\bfk) \to ({\rm Set})\]
assigning to a finitely generated extension $F/\bfk$ the set $\Xcal(F)/\!\sim$ of $F$-valued points of $\Xcal$ modulo isomorphism. A cohomological invariant of $\Xcal$ is a natural transformation
\[\alpha: {\rm Pt}_{\Xcal} \to \M\]
satisfying a \emph{continuity condition} encoding that the natural transformation is well-behaved in regards to specializations. In practice, one should think of a cohomological invariant as a functorial way to assign to a map $f:\Spec(F) \to \Xcal$, an element $\alpha(f) \in \M^{\bullet}(F)$.

The cohomological invariants of $\Xcal$ with coefficients in $\M$ form a graded group $\Inv(\Xcal,\M)$. When $\M=\K_{\rm Mil}$ or $\M=\K_\ell$, the product structure makes $\Inv(\Xcal,\M)$ into a graded ring, and in general $\Inv(\Xcal, \M)$ is an $\Inv(\Xcal, \K_{\rm Mil})$-graded module. Moreover, if~$\M$ is of $\ell$-torsion, then $\Inv(\Xcal, \M)$ is an $\Inv(\Xcal,\K_\ell)$-graded module.

We have a map $\M^{\bullet}(\bfk) \to \Inv(\Xcal,\M)$ which sends an element $x \in \M^{\bullet}(\bfk)$ to the cohomological invariant which assigns to any $f: \Spec(F) \to \Xcal$ the pullback of $x$ to $\M^{\bullet}(F)$. We call the image of this map the subgroup of \emph{constant invariants}. The map is an inclusion whenever $\Xcal(\bfk)$ is not empty.

Below we list some properties of cohomological invariants that we will need:

\begin{enumerate}
    \item A morphism of algebraic stacks $f:\Ycal \to \Xcal$ induces a pullback map $f^*$ on cohomological invariants defined by $f^*(\alpha)(P) = \alpha(f(P))$.
    \item Cohomological invariants form a sheaf in the \emph{smooth-Nisnevich} topology, where coverings are smooth representable maps $\Xcal \to \Ycal$ such that any $F$-valued point of $\Ycal$ lifts to a $F$-valued point of $\Xcal$.
    \item In particular, if $\M=\H_D$ and $\Xcal$ is a smooth stack, then the cohomological invariants with coefficients in $\H_D$ are the sheafification on the smooth-Nisnevich site of $\Xcal$ of twisted cohomology with coefficients in $D$.
    \item If $\Xcal$ is smooth over $\bfk$, a dominant open immersion $\mathcal{U} \to \Xcal$ induces an injective map on cohomological invariants.
    \item If $\Xcal$ is smooth over $\bfk$ and $f:\Ycal \to \Xcal$ is either an affine bundle, the projectivization of a vector bundle or an open immersion whose complement has codimension $\geq 2$ then $f^*:\Inv(\Xcal, \M) \to \Inv(\Ycal, \M)$ is an isomorphism.
    \item If $\Xcal=\left[ X/G\right]$ is the quotient of a smooth scheme by a smooth algebraic group over $\bfk$, then 
    \[
{\rm Inv}^1(\Xcal,\K_\ell)={\rm H}^1_{\textnormal{lis-\'et}}(\Xcal,\mu_\ell) \textnormal{ and } {\rm Inv}^2(\Xcal, {\rm H}_{\mu_\ell^{\vee}})={\rm Br}'(\Xcal)_{\ell},
\]
    where the right hand side is respectively lisse-étale cohomology and the cohomological Brauer group. 
\end{enumerate}

\subsection{Cycle map and negligible cohomology}
Assume the field $\bfk$ contains a primitive $\ell$-th root of $1$, and let $X$ be smooth over $\bfk$. Then, there is a \emph{cycle map} \cite{Mil}*{Ch. VI, Sec. 9}
\[{\rm cl}:{\rm CH}^{\bullet}(X)/\ell \to \H^{\bullet}(X,\ZZ/\ell\ZZ)\]
sending the class of a smooth irreducible subvariety $Z \xrightarrow{i} X$ of codimension $d$ to the element
\[{\rm cl}(Z) = i_*(1_Z) \in \H^{2d}(X,\ZZ/\ell\ZZ).\]
Given a quotient stack $\left[X/G\right]$ where both $X$ and $G$ are smooth over $\bfk$ we can define a cycle map
\[{\rm cl}:{\rm CH}_G^{\bullet}(X)/\ell \to \H^{\bullet}(\left[X/G\right],\ZZ/\ell\ZZ)\]
by means of equivariant approximation \cites{EG, Tot}. Pick a representation $V$ of~$G$ such that the action is free on a subset $U \subset V$ of codimension higher than $d$, so that $\left[(U\times X)/G\right]=(U\times X)/G$ is an algebraic space. By definition, an equivariant cycle $z \in {\rm CH}^d_G(X)$ is represented by a cycle
\[Z \in {\rm CH}^d((X\times U)/G)\stackrel{\rm def}{=}{\rm CH}^d_G(X).\]
Then, we define
\[{\rm cl}(z) = {\rm cl}(Z) \in {\rm H}^{2d}((X\times U)/G,\ZZ/\ell\ZZ)\simeq {\rm H}^{2d}(\left[X/G\right],\ZZ/\ell\ZZ).\]
The definition is well posed thanks to the same double fibration argument that shows that equivariant Chow groups are well defined, and when $\left[ X/G\right]$ is a scheme, it coincides with the classical cycle map.

Following \cite{GMS}*{Sec. 26}, we say that an element $\alpha \in \H^{\bullet}(\Xcal,D)$ is \emph{negligible} if for every field $F$ and map $p:\Spec(F) \to \Xcal$ we have that $p^*(\alpha)=0$. In other words, an element $\alpha$ is negligible if it lies in the kernel of the Smooth-Nisnevich sheafification map
\[\Gamma_{\rm sm-Nis}:{\rm H}^{\bullet}(\Xcal,D) \to \Inv(\Xcal, \H_D).\]
Now let $\Xcal$ be a smooth equidimensional quotient $\left[X/G\right]$ and consider an element $z \in {\rm CH}^d_G(X)$, with $d > 0$. We want to show that ${\rm cl}(z)$ is negligible. Let $U$ be as above and let $\pi$ be the projection $(X\times U)/G \to \left[ X/G \right]$. Then by properties 4, 5 in Sec. \ref{sec:cohinv}, the pullback 
\[\pi^*:\Inv(\left[X/G\right],\H_D) \to \Inv((X\times U)/G,\H_D)\]
is injective. Moreover, by definition $\pi^*({\rm cl}(z))={\rm cl}(Z)$. As $Z$ is supported on a closed subset of positive codimension, there is an open subset $Y \xrightarrow{j} (X\times U)/G$ such that $j^*{\rm cl}(Z)=0$, which shows that 
\[0=\Gamma_{\rm sm-Nis}( j^* \pi^* {\rm cl}(z))=j^*\pi^*\Gamma_{\rm sm-Nis}({\rm cl}(z)).\]
On the other hand, both $\pi^*$ and $j^*$ are injective on cohomological invariants, which implies that $\Gamma_{\rm sm-Nis}({\rm cl}(z))=0$. We have shown:
\begin{prop}\label{prop: negligible}
The image of ${\rm CH}_G^d(X)$ through the cycle map is negligible for $d > 0$. In particular, a cohomology class which produces a noncostant cohomological invariant cannot belong to the $\H^{\bullet}(\bfk,\ZZ/\ell\ZZ)$-submodule generated by the image of the cycle map.
\end{prop}

\subsection{Stiefel-Whitney classes}\label{sec: StiefelW}
Given a finite \'etale morphism $\mathcal{S\to \Xcal}$ of constant degree, we will show how to produce some (possibly trivial) classes in $\Inv(\Xcal, \K_2)$. 

An \emph{\'etale algebra} $E$ over $F$ is a finite $F$-algebra that is isomorphic to $(F')^n$, for some $n$, after passing to some finite extension $F'/F$, i.e. $\Spec(E) \to \Spec(F)$ is a finite \'etale morphism. Passing to the group scheme $T_E(R) = {\rm Aut}_R(E \otimes R)$ we obtain a ${\rm S}_n$-torsor over $F$. The category of finite \'etale algebras of rank~$n$ over~$\bfk$ forms a Deligne-Mumford stack $\textnormal{\'Et}_n$, which is isomorphic to ${\rm BS}_n$. Any finite \'etale morphism $\mathcal{S} \to \Xcal$ of constant degree $n$ induces a morphism $\Xcal \to \textnormal{\'Et}_n$, meaning we can pull the cohomological invariants of $\textnormal{\'Et}_n$ back to $\Xcal$.

The cohomological invariants of $\textnormal{\'Et}_n$ are fully understood \cite{GMS}*{Ch. 7}, and can be described as follows:

Given the ${\rm S}_n$-torsor $\Spec(\bfk) \to {\rm BS}_n$, the Hochschild-Serre Spectral sequence \cite{Mil}*{Ch. III, Thm. 2.20} 
\[{\rm E}_2^{p,q} = \H^p_{\rm grp}({\rm S}_n, \H^q(\bfk,\ZZ/2\ZZ)) \Rightarrow \H^{p+q}({\rm BS}_n,\ZZ/2\ZZ) \]
induces, via the edge map, a morphism 
\[\H^p_{\rm grp}({\rm S}_n,\H^0(\bfk,\ZZ/2\ZZ)) \to \H^{p}({\rm BS}_n,\ZZ/2\ZZ).\]
As $\H^0(\bfk,\ZZ/2\ZZ)=\ZZ/2\ZZ$ this induces a map from the group cohomology ring $\H^{\bullet}_{\rm grp}({\rm S}_n,\ZZ/2\ZZ)$ to $\H^{\bullet}({\rm BS}_n,\ZZ/2\ZZ)$, which becomes an isomorphism over an algebraically closed field.

The image of this map contains elements $\alpha_i, i \in \mathbb{N}$, called the Galois-Stiefel-Whitney classes, whose images through the sheafification map 
\[\H^{\bullet}({\rm BS}_n,\ZZ/2\ZZ) \to \Inv({\rm BS}_n, \H_{\ZZ/2\ZZ})=\Inv({\rm BS}_n, \K_2)\]
generate the cohomological invariants of ${\rm BS}_n$ and $\textnormal{\'Et}_n$ for any choice of cycle module~$\M$ and satisfy the following properties:
\begin{enumerate}
    \item $\Inv(\textnormal{\'Et}_n,\M)=\M^{\bullet}(\bfk) \oplus \alpha_1 \cdot \M^{\bullet}(\bfk) \oplus \ldots \oplus \alpha_{\lfloor n/2\rfloor} \cdot \M^{\bullet}(\bfk),$
    \item $\alpha_0=1$ and $\alpha_i(E)=0$ if $i$ is greater than $\lfloor n/2 \rfloor$.
\end{enumerate}
Moreover, define  the total Galois-Stiefel-Whitney class of $E$ as $\alpha_{\rm tot}(E)=\sum_i \alpha_i(E)$. Then:
\begin{enumerate}
    \item[(3)] $\alpha_{\rm tot}(F^n/F)=1$,
    \item[(4)] $\alpha_{\rm tot}(F(\sqrt{a})/F)=1+\lbrace a \rbrace$,
    \item[(5)] $\alpha_{\rm tot}(E \times E')=\alpha_{\rm tot}(E)\cdot \alpha_{\rm tot}(E').$
\end{enumerate}

Given a finite, degree $n$ \'etale morphism $\mathcal{S} \to \Xcal$, we will write $\alpha_i(\mathcal{S}) \in {\rm Inv}^i(\Xcal,\K_2)$ for the pullback of $\alpha_i \in {\rm Inv}^i(\textnormal{\'Et}_n, \K_2)$ through the morphism induced by $\mathcal{S}$. 

While the properties above are powerful, it is not obvious how to explicitly compute the value of these classes at an \'etale algebra $E/F$. To do so, we will define a different set of generators. Consider the classifying stack ${\rm BO}_n$; it is equivalent to the stack whose objects are $n$-dimensional vector bundles together with a nondegenerate quadratic form. We define cohomological invariants
\[\alpha_i^{\rm SW} \in {\rm Inv}^i({\rm BO}_n,\K_2)\]
and then use them to produce invariants of $\textnormal{\'Et}_n$. 

Given an element $(V_F, q)$ of ${\rm BO}_n(F)$, we can diagonalize the form $q$, yielding an $n$-uple of coefficients $(d_1, \ldots, d_n)$ in~$F^n$. Obviously, picking two different diagonal representations yields different coefficients: it is immediate that an invariant of $q$ must be left unchanged when multiplying coefficients by a square, or when permuting them. We define the $i$-th Stiefel-Whitney class of $q$ as
\[
\alpha^{\rm SW}_i(q) = \sigma_i(d_1,\ldots, d_n) \in \K^i_{2}(F),
\]
where $\sigma_0=1 \in \K^0_2(F)$ and $\sigma_i$ is the $i$-th elementary symmetric polynomial in~$d_1, \ldots, d_n$ seen as an element in $\K_2^i(F)=\H^i(F,\ZZ/2\ZZ(i))$, e.g. 
\[\sigma_2(a,b,c)=\lbrace a , b \rbrace + \lbrace a , c \rbrace + \lbrace b , c \rbrace \in \K^2_2(\bfk(a,b,c)).\]

For any cycle module module $\M$ we have
\[
\Inv({BO}_n,\M)=\M^{\bullet}(\bfk) \oplus \alpha^{\rm SW}_1 \cdot \M^{\bullet}(\bfk)_2 \oplus \ldots \oplus \alpha^{\rm SW}_{n} \cdot \M^{\bullet}(\bfk)_2.
\]

Given $E/F \in \textnormal{\'Et}_n(F)$, we can define a nondegenerate scalar product on $E$, seen as a $F$-vector space, by $\langle x, y \rangle_{\rm tr} = {\rm tr}(\cdot xy)$, that is the couple $x,y$ is sent to the trace of the multiplication by $xy$ seen as an endomorphism of the $F$-vector space~$E$. Write ${\rm tr}_{E/F}$ for the corresponding quadratic form. Then, we define 
\[\alpha_i^{\rm SW}(E) \stackrel{\rm def}{=} \alpha_i^{\rm SW}({\rm tr}_{E/F}).\] 
We have
\[
\Inv(\textnormal{\'Et}_n,\M)=\M^{\bullet}(\bfk) \oplus \alpha^{\rm SW}_1 \cdot \M^{\bullet}(\bfk)_2 \oplus \ldots \oplus \alpha^{\rm SW}_{\lfloor n/2\rfloor} \cdot \M^{\bullet}(\bfk)_2. 
\]
Note that the map $\Inv({BO}_n,\M) \to \Inv(\textnormal{\'Et}_n,\M)$ is not injective as the image of~$\alpha^{\rm SW}_i$ is zero for $i > \lfloor n/2 \rfloor + 1$ and the image of $\alpha^{\rm SW}_{\lfloor n/2 \rfloor + 1}$ is a combination of lower degree classes.

Galois-Stiefel-Whitney classes and Stiefel-Whitney classes are tied by the following simple relation \cite{GMS}*{Thm.25.10}:
\[
\alpha_i=
\begin{cases}
\alpha_i^{\rm SW} \quad & i \textnormal{ odd},\\
\alpha_i^{\rm SW} + \lbrace 2 \rbrace \cdot \alpha^{\rm SW}_{i-1}\quad & i\textnormal{ even}.
\end{cases}
\]
\begin{ex}
 If we write $E_2=\bfk(\sqrt{a_1},\sqrt{a_2})/\bfk(a_1,a_2)$, we know by \cite{DilPirM3}*{Ex. 3.12} that  
\[\alpha_{\rm tot}(E_2)=1 + \lbrace a_1, a_2 \rbrace + \lbrace -1, a_1a_2\rbrace.\]
As an additional example we will compute the total Stiefel-Whitney class of
\[E_3 = \bfk(\sqrt{a_1},\sqrt{a_2}, \sqrt{a_3})/\bfk(a_1,a_2, a_3).\]
 Consider the basis $1,\sqrt{a_1}, \sqrt{a_2}, \ldots, \sqrt{a_2a_3}, \sqrt{a_1a_2a_3}$ of $E_3$ as a vector space over $\bfk(a_1,a_2,a_3)$. In this basis the trace form is described by the diagonal matrix 
\[{\rm diag}(8, 8a_1, \ldots, 8a_2a_3, 8a_1a_2a_3)\sim {\rm diag}(2, 2a_1, \ldots, 2a_2a_3, 2a_1a_2a_3).\] 
This shows that 
\[ \alpha^{\rm SW}_i(E_3)=\sigma_i(2, 2a_1, \ldots, 2a_2a_3, 2a_1a_2a_3)\]
So we get
\begin{align*}
\alpha_1^{\rm SW}(E_3) &= \lbrace 2 \rbrace + \lbrace 2a_1 \rbrace + \ldots + \lbrace 2a_1a_2a_3 \rbrace = 8\lbrace 2 \rbrace + 4(\lbrace a_1 \rbrace + \lbrace a_2 \rbrace + \lbrace a_3 \rbrace)=0,\\
\alpha_2^{\rm SW}(E_3) &= \lbrace 2, 2a_1 \rbrace + \ldots +\lbrace 2a_2a_3, 2a_1a_2a_3 \rbrace = \ldots = 0, \\
\alpha_3^{\rm SW}(E_3) &= \lbrace 2, 2a_1, 2a_2 \rbrace + \ldots +\lbrace 2a_1a_3, 2a_2a_3, 2a_1a_2a_3 \rbrace = \ldots = 0, \\
\alpha_4^{\rm SW}(E_3) &= \lbrace 2, 2a_1, 2a_2, 2a_3 \rbrace + \ldots +\lbrace 2a_1a_3, 2a_2a_3, 2a_1a_2a_3, 2a_1a_2a_3 \rbrace\\
&\mkern-60mu =5\lbrace 2 \rbrace \cdot \sigma_3(a_1, \ldots, a_1a_2a_3) + \sigma_4(a_1, \ldots, a_1a_2a_3)\\
&\mkern-60mu =\lbrace -1, a_1, a_2, a_3 \rbrace + \lbrace -1 \rbrace^2 \cdot(\lbrace a_1, a_2 \rbrace + \lbrace a_1, a_3 \rbrace + \lbrace a_2, a_3 \rbrace) + \lbrace -1 \rbrace^3\cdot \lbrace a_1a_2a_3\rbrace.
\end{align*}
These are all the classes we need since $\alpha_i(E)=0$ for $i > \lfloor {\rm deg}(E)/2 \rfloor$. 

Applying the relation between the Stiefel-Whitney and Galois-Stiefel-Whitney classes we immediately get that that the total Galois-Stiefel-Whitney class $\alpha_{\rm tot}(E_3)$ is equal to 
\[1 + \lbrace -1, a_1, a_2, a_3 \rbrace + \lbrace -1 \rbrace^2 \cdot(\lbrace a_1, a_2 \rbrace + \lbrace a_1, a_3 \rbrace + \lbrace a_2, a_3 \rbrace) + \lbrace -1 \rbrace^3\cdot \lbrace a_1a_2a_3\rbrace.\]
Section \ref{sec: StiefEn} will be dedicated to computing these classes in general.
\end{ex}

\subsection{Curves of compact type}

Recall that $\Mcal_g$, the moduli stack of smooth genus $g$ curves, is the algebraic stack whose objects are proper smooth morphisms $C \to S$ such that each geometric fiber is a smooth curve of genus $g$. Its standard compactification is the stack $\overline{\Mcal_g}$, the moduli stack of stable curves of genus $g$, whose objects are proper flat morphisms $X \to S$ such that every geometric fiber is a stable curve of genus $g$, that is a proper one-dimensional scheme with arithmetic genus $g$, at worst nodal singularities, and finite automorphism group.

We have a partial compactification in between the two: the stack  of genus $g$ \emph{compact type} curves \cite{LandProper}
\[\Mcal_g \subset \Mcal_g^{\rm ct} \subset \overline{\Mcal_g}\]
whose objects are proper flat morphisms $C \to S$ such that each geometric fiber is a stable curve and its dual graph is a tree. Equivalently, a curve $C$ over a field $F$ is of compact type if its Picard schemes ${\rm Pic}^d(C)$ are proper for all $d$.

Now let $\mathcal{A}_g$ be the moduli stack of principally polarized abelian varieties (p.p.a.v.) of dimension $g$. The torelli morphism
\[\Mcal_g \to \mathcal{A}_g, \quad (C \to S) \mapsto ({\rm Jac}(C), \theta(C))\]
sending a smooth curve of genus $g$ to its Jacobian, has a natural extension to $\Mcal_g^{\rm ct}$. If $C$ is a compact type curve over a field $F$ and $C_1 \sqcup \ldots \sqcup C_r$ is its normalization then
\[{\rm Jac}(C)={\rm Jac}(C_1) \times \ldots \times {\rm Jac}(C_r), \quad \theta(C) = \theta(C_1) + \ldots + \theta(C_r).\]

\subsection{Theta characteristics}

Let $C$ be a curve over a field $F$. A \emph{theta characteristic} \cite{Mum71} of $C$ is an element $\sigma \in {\rm Pic}^{g-1}(C)$ such that $2\sigma = \omega_C \in {\rm Pic}^{2g-2}(C)$. Given two theta characteristics $\sigma, \sigma'$, clearly $\sigma-\sigma'$ is a $2$-torsion element. Conversely, for any $2$-torsion element $\alpha \in {\rm Jac}(C)_2$, the element $\sigma + \alpha$ is a theta characteristic. This shows that the subset of theta characteristics $S_C \subset {\rm Pic}^{g-1}(C)$ is a zero-dimensional subscheme and a ${\rm Jac}(C)_2$-torsor. In particular, its degree is $2^{2g}$.

A theta characteristic $\sigma$ is \emph{even} if ${\rm dim}(\H^0(C,\sigma))$ is even, and odd otherwise. The subschemes $S^{+}_C$ and $S^{-}_C$ of even and odd theta characteristics have degrees $2^g(2^{g-1}+1)$ and $2^g(2^{g-1}-1)$, respectively. Now, consider the universal relative Picard scheme
\[{\rm Pic}^{g-1}(\mathcal{C}_g) \to \Mcal^{\rm ct}_g.\]
Taking respectively the substack of even, odd and all theta characteristics, we get the finite \'etale morphisms
\[\mathcal{S}^+_g \to \Mcal_g^{\rm ct}, \quad \mathcal{S}^-_g \to \Mcal_g^{\rm ct}, \quad \mathcal{S}_g \to \Mcal_g^{\rm ct},\]
of degrees $2^g(2^{g-1}+1), 2^{g}(2^{g-1}-1)$ and $2^{2g}$ respectively. Given a curve $C$, we will write $S^+_C, S^-_C$ and $S_C$ for the \'etale algebras induced by its respective theta characteristics. We will also consider the \'etale covering of degree $2^{2g}$
\[(\Xcal_{g})_2 \to \mathcal{A}_g\]
given by the $2$-torsion in the universal abelian variety. Note that pulling back $(\Xcal_{g})_2$ to $\Mcal_g$ we obtain $(\mathcal{J}_g)_2$, the $2$-torsion in the universal Jacobian. When a curve $C/F$ has a theta characteristic defined over $F$, the ${\rm Jac}(C)_2$-torsor $S_C$ is trivial, which shows that ${\rm Jac}(C)_2=\Spec(S_C)$.

\subsection{Hyperelliptic curves}

A hyperelliptic curve of genus $g$ over a field $F$ is a couple $(C,\iota)$ where $C$ is a smooth genus $g$ curve over $F$ and $\iota$ is an involution of $C$ such that the quotient $Q=C/\iota$ has genus zero. The two-sheeted covering $C\to Q$ ramifies on the Weierstrass divisor $W_C \subset C$ which is \'etale of degree $2g+2$ over $F$.

A family of hyperelliptic curves is a proper, smooth morphism of relative dimension one such that every geometric fiber is an hyperelliptic curve. The stack $\Hcal_g$ whose objects are families of hyperelliptic curves is smooth over $\bfk$ of dimension $2g-1$, and the forgetful map $\Hcal_g \to \Mcal_g$ is a closed immersion. The Weierstrass divisor induces an \'etale cover
\[\mathcal{W}_g \to \Hcal_g\]
of degree $2g+2$, where the objects of $\mathcal{W}_g$ are families of hyperelliptic curves plus a section of the Weierstrass divisor. There is a second presentation of the stack of hyperelliptic curves by Arsie and Vistoli \cite{ArsVis}*{Rmk. 3.3, Ex. 3.5} that we will need. Consider the fibered category $\Hcal'_{\rm sm}(1,2,g+1)$ defined by
\[\Hcal'_{\rm sm}(1,2,g+1)(S)=\lbrace P \to S, L , s \rbrace,\]
where $P\to S$ is a Brauer-Severi scheme of relative dimension one, $L$ is an invertible sheaf on $P$ whose restriction to each geometric fiber has degree $-g-1$ and $s$ is a section of $L^{\otimes -2}$ whose zero locus is \'etale on $S$. Then the morphism
\[
\Hcal_g \to \Hcal'_{\rm sm}(1,2,g+1)
\]
obtained as in \cite{DilPirRS}*{Prop. 2.1}, is an isomorphism. Note that the curve $(C,\iota)$ maps to the triple $(C/\iota, \pi_*\mathcal{O}_C, s)$ where the vanishing locus of $s$ is exactly the image of $W_C$, and $\pi\colon C\xrightarrow{}C/\iota$ is the quotient map.

\subsection{Theta characteristics of hyperelliptic curves}\label{sec: hypertheta}

Consider a hyperelliptic curve $C$ over an algebraically closed field $F$ of characteristic different from $2$. Following \cite{Dolg}*{Ch. 5} we get the following description of the 2-torsion in its Jacobian. Let $c_1, \ldots, c_{2g+2}$ be the $2g+2$ points in the Weierstrass divisor, and $p_1, \ldots, p_{2g+2}$ their images. 
We have $(2g-2)c_i = \omega_C$, so $(g-1)c_i$ is a theta characteristic for any~$i$. Moreover, $2c_i - 2c_j \sim 0$ as it is the pullback of $p_i - p_j \in {\rm Pic}(\PP^1_F)$. 

Set $I=\lbrace 1,\ldots,  2g+2 \rbrace$, and let $T$ be any subset of $I$. We define
\[
\alpha_T = \sum_{i \in T}( c_i - c_{2g+2}) \in {\rm Pic}^0(C)\left[2\right].
\]
Note that $\alpha_I \sim 0$ as it is the divisor of a rational function on $C$ \cite{Dolg}*{Eq. 5.12}. This shows that we have $\alpha_T \sim \alpha_{I\smallsetminus T}$ and thus we can always pick $| T | \leq g+1$. Finally, we can also assume that $|T| \equiv g+1 \mod(2)$ by adding or removing $c_{2g+2}$ to~$T$, which does not change the element $\alpha_T$. Every element in ${\rm Pic}^0(C)\left[2\right]$ is uniquely represented by an element $\alpha_T$ as such unless $|T|$ is exactly $g+1$, in which case $\alpha_T=\alpha_{I\smallsetminus T}$ and there are exactly two subsets of $I$ representing the same element.

Now, we know that $(g-1)c_{2g+2}$ is a theta characteristic of $C$, which immediately implies that every theta characteristic can be written as
\[\theta_T = \alpha_T + (g-1)c_{2g+2},\]
for some $T$ as above. Then $\theta_T$ is even if and only if $|T| \equiv g+1 \mod(4)$.

\section{Curves of genus three}

We begin by studying the case $g=3$. This will allow us to compute the restriction map from the Brauer group of $\Mcal_3$ to the Brauer group of $\Hcal_3$. 

Recall the description of the Brauer group of $\Hcal_3$ from \cite{DilPirBr}*{Thm. A}:
\[{\rm Br}(\Hcal_3)_{\ell}={\rm Br}(\bfk)_{\ell} \oplus \H^1(\bfk,\ZZ/28\ZZ)_{\ell} \oplus \alpha_2(\mathcal{W}_3)\cdot \ZZ/2\ZZ \oplus w_2 \cdot \ZZ/2\ZZ.\]
Here the $\H^1(\bfk,\ZZ/28\ZZ)$ component comes from the cap product 
\[{\rm Pic}(\Hcal_3) = \ZZ/28\ZZ = \H^1(\Hcal_3, \mu_{28}) \otimes \H^1(\bfk,\ZZ/28\ZZ)\to \H^2(\Hcal_3,\mu_{28}),\]
the element $\alpha_2(\mathcal{W}_3)$ is the second Galois-Stiefel-Whitney class coming from the Weierstrass divisor; note that the two-torsion submodule 
\[{\rm H}^1(\bfk,\ZZ/28\ZZ)_2 = {\rm H}^1(\bfk,\ZZ/2\ZZ)\]
is generated by $\alpha_1(\mathcal{W}_3)$. The element $w_2$ comes from ${\rm BPGL}_2$ and sends a curve~$C$ to the class $\left[ C/\iota\right]$ in the Brauer group.  Another way to see the element $w_2$ is through the isomorphism ${\rm PGL}_2 \simeq {\rm SO}_3$. Given a conic $Q$ cut out by a normalized quadratic form $q\in \mathcal{O}_{\PP^2_F}(2)$ we have
\[w_2(Q) = \alpha_2^{\rm SW}(q)-\alpha_2^{\rm SW}(q_0)\]
where $q_0$ is any normalized form defining a rational conic. For example, if 
\[q_0 = -x_0^2-x_1^2+x_2^2,\]
then, an explicit computation yields
\[w_2(Q) = \alpha_2^{\rm SW}(q)+\lbrace -1, -1 \rbrace.\]
We know that the pullback of $\alpha_2(S) \in {\rm Br}(\Mcal_3)$ is equal to
\[r \cdot \alpha_1(\mathcal{W}_3) + s \cdot \alpha_2 (\mathcal{W}_3) + t \cdot w_2, \]
where $r \in \H^1(\bfk,\ZZ/14\ZZ)$ and $s,t \in \ZZ/2\ZZ$. We will construct two curves $C, C'$ to test respectively $s,t$ and $r$.

\subsection{First test curve}
For our first test curve $C/F$ we want the sugroup 
\[{\rm Br}(\bfk) \oplus {\rm H}^1(\bfk,\ZZ/2\ZZ) \oplus \alpha_2(\mathcal{W}_3)\cdot \ZZ/2\ZZ\]
to map injectively to ${\rm Br}(F)$.

Set $F=\bfk(a_1,a_2,a_3)$, $E_i = F(\sqrt{a_i}), E_{ij}=F(\sqrt{a_i},\sqrt{a_j}), K=F(\sqrt{a_1},\sqrt{a_2},\sqrt{a_3})$. Consider the points
\[
\lbrace p_1, \ldots, p_6, q_7, q_8 \rbrace=\lbrace (\pm\sqrt{a_i}:1), 0, \infty \rbrace \subset \PP^1_K.
\]
The action of ${\rm Gal}(K/F)$ exchanges the points $p_{2i-1}$ and $p_{2i}$. There is clearly a divisor $D \subset \PP^1_F$ whose inverse image in $\PP^1_K$ is $p_1 + \ldots + q_8$. Let $(C,\iota)$ be an hyperelliptic curve over $F$ such that $C/\iota=\PP^1_F$ and the image of the Weierstrass locus is $D$. We will abuse notation and denote the points in the Weierstrass divisor with the same names as their images.

As the extension $F/\bfk$ is purely transcendental we have ${\rm Br}(\bfk) \subset {\rm Br}(F)$. The \'etale algebra given by the Weierstrass divisor is 
\[W_C = E_1 \times E_2 \times E_3 \times F^2,\]
which shows that
\[\alpha_1(W_C)=\lbrace a_1a_2a_3 \rbrace,\, \alpha_2(W_C) = \lbrace a_1, a_2 \rbrace + \lbrace a_1, a_3 \rbrace +\lbrace a_2, a_3 \rbrace. \]
Given an element 
\[x = x_0 + \alpha_1(\mathcal{W}_3) \cdot x_1 + \alpha_2(\mathcal{W}_3) \cdot x_2,\]
with $x_0 \in {\rm Br}(\bfk), x_1 \in \H^1(\bfk, \ZZ/2\ZZ)$ and $x_2 \in \ZZ/2\ZZ$, then the pullback $x_C \in {\rm Br}(F)$ is equal to
\[x_0 + \lbrace a_1a_2a_3 \rbrace \cdot x_1 + \left(\lbrace a_1, a_2 \rbrace + \lbrace a_1, a_3 \rbrace +\lbrace a_2, a_3 \rbrace\right) \cdot x_2.\]
We have that $\partial_{a_1=0}(x_C) = x_1 + \lbrace a_2a_3 \rbrace \cdot x_2$ and $\partial_{a_1=0}(\partial_{a_2=0}(x_C))=x_2$, which immediately implies that for $x_C$ to be zero we must have $x_2=x_1=x_0=0$.

We now turn to the theta characteristics of $C$. By the discussion in \ref{sec: hypertheta}, the set of odd theta characteristics of $C_K$ is equal to the set 
\[\lbrace\theta_T \mid |T| =2 \rbrace \]
of these 
\begin{itemize}
    \item The three corresponding to couples of conjugate points and the one corresponding to $\lbrace q_7, q_8\rbrace$ are defined over $F$.
    \item The twelve corresponding to a $p$ point and a $q$ point are defined over the field of definition of the $p$ point.
    \item The twelve corresponding to a pair of non-conjugated $p$ points are defined over the corresponding degree four extension.
\end{itemize}

We conclude that
\[S^-_C=F^4 \times E_1^2 \times E_2^2 \times E_3^2 \times E_{12} \times E_{13} \times E_{23}.\]

\subsection{Second test curve}
For the second curve $C'$ we want 
\[\alpha_1(W_{C'})=\alpha_2(W_{C'})=0,\; w_2(C')\neq0.\]

For any base field $\bfk$ of characteristic different from $2$, we can pick an irrational smooth conic $Q$ over $F=\bfk(a_1,a_2)$. For example, we can pick the quadratic form
 \[q(x,y,z)=a_1x^2 + a_2y^2 + a_1^{-1}a_2^{-1}z^2,\]
 which is nontrivial as $\alpha_2^{\rm SW}(q)=\lbrace a_1,a_2 \rbrace + \lbrace -1, a_1a_2 \rbrace$. The conic $Q$ splits over a degree two extension $F'/F$, e.g. $F'=F(\sqrt{-a_2})$ and in particular it has infinitely many $F'$-valued points.
\begin{rmk}
 As a sanity check, computing the class $w_2(C'_{F'})$ yields:
 \begin{align*}
   w_2(C'_{F'}) =& \alpha_2^{\rm SW}(q) + \lbrace -1, -1\rbrace \\&=\lbrace a_1,a_2 \rbrace + \lbrace -1, a_1a_2 \rbrace + \lbrace -1, -1\rbrace\\&=
 \lbrace a_1, -(\sqrt{-a_2})^2\rbrace + \lbrace -1, a_1 \rbrace + \lbrace -1, -(\sqrt{-a_2})^2\rbrace + \lbrace -1, -1 \rbrace\\&=
 \lbrace a_1, -1 \rbrace + \lbrace a_1, -1 \rbrace + \lbrace -1, -1 \rbrace + \lbrace -1, -1 \rbrace=0,  
 \end{align*}
 as expected.
\end{rmk}

 Its Picard group is generated by the class $\left[ x \right]$ of any $F'$-valued point. Pick four more $F'$ valued points $x_1,\ldots, x_4$. There is a section $s$ of $\mathcal{O}_Q(4x)$ whose zero locus is exactly $x_1,\ldots, x_4$. Then the triple
 \[(Q/F, \mathcal{O}_Q(-2x), s)\]
 is a genus three smooth hyperelliptic curve over $C'/F$ with 
 \[w_2(C')=\lbrace a_1,a_2 \rbrace + \lbrace -1, a_1a_2 \rbrace.\] 
 The Weierstrass divisor is given by four copies of $\Spec(F')$, so that
 \[\alpha_1(W_{C'})=\alpha_2(W_{C'})=0.\]
  Write $p_4, p'_4$ for the fiber of $x_4$ in $C'_{F'}$. Every theta characteristic on $C'_{F'}$ can be written as $\alpha_T + 2p_4$. Note that as $2p'_4-2p_4\sim 0$ we have 
  \[\alpha_T + 2p_4 = \alpha_T + 2p'_4,\]
  which shows that the field of definition of $\theta_T$ is the same as the field of definition of $\alpha_T$. Repeating the computation above, we get 
 \[S^-_{C'}=F^4 \times (F')^{12}.\]

\subsection{Computation of the pullback}
Our setup will easily allow us to show the following:

\begin{thm}
The pullback of $\alpha_2(S) \in {\rm Br}(\Mcal_3)$ to ${\rm Br}(\Hcal_3)$ is equal to
\[
\alpha_{2}(\mathcal{W}_3) + \lbrace -1\rbrace \alpha_1(\mathcal{W}_3).
\]
In particular the map $i^*:{\rm Br}(\Mcal_3) \to {\rm Br}(\Hcal_3)$ is injective.
\end{thm}
\begin{proof}

Recall that $i^*\alpha_2(S)=r \cdot \alpha_1(\mathcal{W}_3) + s \cdot \alpha_2 (\mathcal{W}_3) + t \cdot w_2$ for some $r \in \H^1(\bfk,\ZZ/14\ZZ)$ and $s,t \in \ZZ/2\ZZ$. We begin by computing $r$ and $s$:

First we focus on the first test curve $C$. As by construction $C/\iota=\PP^1_F$ we have that~$w_2(C)=0$.
We showed that the subgroup 
\[ {\rm Br}(\bfk) \oplus {\rm H}^1(\bfk,\ZZ/2\ZZ) \oplus \alpha_2(\mathcal{W}_3)\cdot \ZZ/2\ZZ\]
injects to ${\rm Br}(F)$, with 
\[\alpha_1(W_C)=\lbrace a_1a_2a_3 \rbrace,\, \alpha_2(W_C) = \lbrace a_1, a_2 \rbrace + \lbrace a_1, a_3 \rbrace +\lbrace a_2, a_3 \rbrace.\]
Moreover we know that
\[S^-_C=F^4 \times E_1^2 \times E_2^2 \times E_3^2 \times E_{12} \times E_{13} \times E_{23}.\]
Then an immediate computation gives us
\[\alpha_2(S^-_C)=\lbrace -1, a_1a_2a_3 \rbrace + \lbrace a_1, a_2 \rbrace+ \lbrace a_1, a_3 \rbrace + \lbrace a_2, a_3 \rbrace\]
proving that $r=\lbrace -1 \rbrace$ and $s=1$.

Now we want to prove that $t=0$. We know that 
\[w_2(C')=\alpha_2^{\rm SW}(Q)\neq 0.\]
Moreover $\alpha_1(W_{C'})=\alpha_2(W_{C'})=0$, which shows that we have $t=0$ if and only if~$\alpha_2(S_{C'})=0$. The latter is clear as $S_{C'}=F^4 \times (F')^{12}$.
\end{proof}

\subsection{Even theta characteristics}

As a toy case for the general computation in the next section we now consider the even theta characteristics of $C$. These correspond to the subsets $T$ such that $| T|  \equiv 0\mod(4)$. Moreover it is relevant to note that if $| T | = 4$ then $T$ induces the same theta characteristic as its complement. We have
\begin{itemize}
    \item The empty set, which is defined over $F$.
    \item The three sets corresponding to the two rational points and two conjugated points, which are defined over $F$.
    \item The twelve sets corresponding to two conjugate points, one non rational point and $q_8$. Taking their complements we get the twelve sets corresponding to two conjugate points, one non rational point and $q_7$.
    \item The twelve sets corresponding to the two rational points and two non-conjugated points, which are defined over the corresponding degree four extension. Taking their complements, we obtain the twelve sets corresponding to two conjugated points and two non-conjugated points, which are defined over the corresponding degree four extension.
    \item The eight sets corresponding to three non-conjugated points and $q_7$, which are defined over the degree eight extension. Taking their complements we obtain the eight sets corresponding to three non-conjugated points and $q_8$.
\end{itemize}

This shows that 
\[S^{+}_C = F^4  \times \prod_{i} E_i^2 \times \prod_{i< j} E_{ij} \times K. \]

We saw in Section \ref{sec: StiefelW} that
\[\alpha_{\rm tot}(K)=1+\lbrace -1 , a_1, a_2, a_3\rbrace + \lbrace -1 \rbrace^2(\lbrace a_1,a_2 \rbrace + \lbrace a_1, a_3 \rbrace + \lbrace a_2, a_3 \rbrace) + \lbrace -1 \rbrace^3\lbrace a_1a_2a_3\rbrace\]
and thus
\[\alpha_{\rm tot}(S^{+}_C)=\alpha_{\rm tot}(S^{-}_C)\alpha_{\rm tot}(K)=1 + \lbrace a_1,a_2 \rbrace + \lbrace a_1, a_3 \rbrace + \lbrace a_2, a_3 \rbrace +  \]
\[\lbrace -1 \rbrace^7\lbrace a_1 a_2 a_3 \rbrace + \lbrace -1 \rbrace^6(\lbrace a_1,a_2 \rbrace + \lbrace a_1, a_3 \rbrace + \lbrace a_2, a_3 \rbrace) + \lbrace -1 \rbrace^5\lbrace a_1, a_2, a_3 \rbrace\]
plus terms of higher degree.

\section{Theta characteristics in higher genus}
\subsection{Our test curve}

Consider the field $F=\bfk(a_1,\ldots,a_g)$, its field extension $K=F(\sqrt{a_1},\ldots, \sqrt{a_g})$ and points

\[p_1, p_2, \ldots p_{2g-1}, p_{2g}, q_{2g+1}, q_{2g+2} \in \PP^1_K\]
such that for all $j \leq g$ the points $p_{2j-1}, p_{2j}$ are defined over $F(\sqrt{a_j})$ and are conjugated under the action of ${\rm Gal}(F(\sqrt{a_j})/F)=\ZZ/2\ZZ$, and $q_{2g+1}$ and $q_{2g+2}$ are defined over $F$. 

The divisor 
\[D_K=\sum_{j\leq 2g}p_j + q_{2g+1} + q_{2g+2} \subset \PP^1_K\] 
descends to a divisor $D$ on $\PP^1_F$ and thus we can define an hyperelliptic curve $C/F$ whose quotient by the hyperelliptic involution is $\PP^1_F$ and such that the image of $W_C$ is equal to $D$. Again, we will use the same names for points in the Weierstrass divisor and their images.

Given any subset $A \subseteq \lbrace 1, \ldots, g\rbrace$ we write $E_A$ for the extension of $F$ generated by the square roots of $a_s$ for all $s \in A$.

\begin{lm}\label{lm:def_field}
A theta characteristic $\theta_T$ of $C_K$ has minimal field of definition $E_A$ if and only if for every $a \in A$, exactly one of $2a-1, 2a$ belongs to $T$.
\end{lm}
\begin{proof}
Consider the subgroup ${\rm Gal}(K/E_A)$ 
of ${\rm Gal}(K/F)$. It has order $2^{g-\mid A \mid}$ and it is generated by $\lbrace\sigma_t \mid t\notin A\rbrace$ where $\sigma_t$ swaps $\pm \sqrt{a_t}$ and fixes $\sqrt{a_s}$ for $s \neq t$. We can write
\[\theta_T = \sum_{s \in A} (x_s - q_{2g+2})  + \sum_{t \notin A} (p_t + p'_t - 2q_{2g+2}) + Q,\]
where $x_a$ is either $p_{a}$ or $p'_{a}$ and $Q$ is a combination of $q_{g+1}, q_{g+2}$. It is then immediate that for any $t \notin A$ the element $\sigma_t$ fixes $\theta_T$, which shows that the field of definition of~$\theta_T$ must be contained in $E_A$.
Finally, consider the possible conjugates of $\theta_T$ through the action of ${\rm Gal}(K/F)$. Obviously each $x_s$ can independently be acted upon to be either $p_{2s-1}$ or $p_{2s}$, and each choice gives a different element, which shows that $\theta_T$ has at least $2^{\mid A \mid}$ conjugates, so its field of definition must be exactly~$E_A$,  concluding our proof. 
\end{proof}

Define $n^-_A, n^+_A$ as the number of times the field $E_A$ appears in the étale algebras $S^{-}_C$ and $S^+_C$, respectively. By Lemma \ref{lm:def_field} we must have 
\[\sum_{A\subset \lbrace 1,\ldots,g \rbrace} 2^{|A|}n^-_A = 2^{g-1}(2^g -1), \quad \sum_{A\subset \lbrace 1,\ldots,g \rbrace} 2^{|A|}n^+_A = 2^{g-1}(2^g +1).\]
We claim the following equality holds:

\begin{prop}
In the situation described above, we have $n^-_A=2^{g-1-|A|}$ for $|A| < g$, and for $A=\lbrace 1, \ldots, g\rbrace$ we have $n^-_A=0$.    
\end{prop}
\begin{proof}
    Note that when $A=\lbrace 1, \ldots, g\rbrace$ obviously $n^-_{A}=0$ as we can always pick $|T| \leq g+1$ and when $|T|=g$ the class is even.
To prove the proposition, first we need to understand, given a theta characteristic $\theta_T$ how to generate a new theta characteristic $\theta_{T'}$ with the same field of definition that is not in the same conjugacy class as $\theta_T$. The most obvious way is to add elements to $T$, so that $T' \supset T$. We must have $|T| \equiv |T'| \mod(4)$, and their stabilizers must be the same. Thus the smallest sets we can add are
\begin{itemize}
    \item Two couples of conjugate points $p_{2i-1},p_{2i}$ and $ p_{2j-1}, p_{2j}$.
    \item Two conjugate points $p_{2i-1},p_{2}$ and the two rational points $q_{2g+1}, q_{2g+2}$.
\end{itemize}
Now we want to count the ways we can extend a \emph{minimal} such $\theta_T$, i.e. such that no smaller set $T$ can realize the field of definition $E_A$. Write $|A|=r$, and consider for simplicity the case where $g \equiv 3 \mod(4)$, so that $|T|\equiv 2 \mod(4)$. We reason based on the remainder of $r$ mod $4$:

\begin{enumerate}
    \item[$r \equiv 2$:] A minimal $T$ has exactly $r$ elements, and for each extension $T' \supset T$ we have to pick either a couple of conjugate points or the two rational points, so we are picking between $g+1-r$ choices. The number of choices we have to make to get an extension is even, and the maximum order of $T$ is $g-1$, so we have the equation $r+2i \leq g-1 \sim i \leq (g-1-r)/2$. In other words, any subset of $g+1-r$ elements with an even number of elements and with less than half of the total elements (note that $(g+1-r)/2$ is odd) will give us a different, non conjugate theta characteristic with the same field of definition. Restricting to even order subsets halves the choices, and restricting to subsets with less than half of the total elements halves them again, so we have a total of $2^{g-r+1}/4=2^{g-r-1}$ possibilities.

    \item[$r \equiv 0$:] A minimal $T$ must have $r+2$ elements, so one for each element of $I$ plus either two conjugate points or the two rational points, for which we have $g-r+1$ choices. Then to get an extension we have to make the same choice an even number of times, so we added $2(2j+1)$ elements. We must have $r+2(2j+1) \leq g-1 \sim 2j+1 \leq (g-1-r)/2$. These correspond to the subsets of $g+1-r$ elements with an odd number of elements and less than half the total elements (note that $(g-r+1)/2$ is even). So again we have a total of $2^{g-r-1}$ possibilities.
    
    \item[$r \equiv 1$:] A minimal $T$ must have $r+1$ elements, so we must necessarily add one of the two rational points. Then we can do the same reasoning as in the case $r \equiv 2 \mod(4)$ but we start with $g-r$ possible choices, and following the same reasoning as for $r\equiv 2 \mod(4)$ we are counting the subsets of $g-r$ elements of even order with less than half the elements. Including the first choice between the two rational points we get exactly $2\cdot 2^{g-r-2}=2^{g-r-1}$ possibilities.
    
    \item[$r \equiv 3$:] A minimal $T$ must have $r+3$ elements, so we must add one of the two rational points and two conjugate points to begin with. Following the same reasoning as for $r\equiv 0 \mod(4)$ we are counting the subsets of $g-r$ elements of odd order and with less than $(g-r-1)/2$ elements, so counting the initial choice we get $2\cdot 2^{g-r-2}=2^{g-r-1}$ possibilities.
\end{enumerate}
The same reasoning works, up to some permutation of the cases, for all other $\mod (4)$ equivalence classes of $g$. This proves that for all $|A|<g$ we have that $n^-_{A}$ is at least $2^{g-r-1}$. But then adding up the degrees of all the copies of $E_A$ we found we get $2^{g-1}(2g-1)={\rm deg}(S^-_C/F)$, so the opposite inequality must also hold, proving our claim.
\end{proof}

\begin{prop}
In the situation described above, we have $n^+_A=2^{g-1-|A|}$ for $|A| < g$, and for $A=\lbrace 1, \ldots, g\rbrace$ we have $n^+_A=1$.    
\end{prop}
\begin{proof}
    The main difference to keep in mind here is the following: if $|T|=g+1$, then $\theta_T = \theta_{T^{\rm c}}$, so the number of extensions where this happens has to be halved. On the other hand, observe that if $n$ is even/odd the subsets of a set of $2n$ elements that have an even/odd order modulo taking the complement are exactly $2^{n-2}$, and they are in correspondence with the even/odd sets with at most $n-2$ elements plus half of the even/odd sets with $n$ elements as those need to be counted together with their complement. Using this we can repeat the same reasoning as above to prove our claim.
\end{proof}

\subsection{Stiefel-Whitney classes of multiquadratic extensions}\label{sec: StiefEn}
From here on, for simplicity, we will write
\[\varepsilon = \lbrace -1 \rbrace \in \K^{\bullet}_{\rm Mil}(\bfk).\]

The aim of this section is to compute the Galois-Stiefel-Whitney classes of the extensions $E_A /F$. By functoriality, it is clearly sufficient to compute the Galois-Stiefel-Whitney classes of the \'etale algebras
\[E_n/F_n = \bfk(\sqrt{a_1}, \ldots, \sqrt{a}_n)/\bfk(a_1,\ldots,a_n).\]
If we pick the basis $1, \sqrt{a_1}, \ldots , \sqrt{a_n}, \sqrt{a_1a_2}, \ldots, \sqrt{a_1\ldots a_n}$ the trace form $q_{E_n}$ is defined by the diagonal matrix 
\[2^n{\rm diag}(1,a_1, \ldots, a_1\ldots a_n)\sim 
\begin{cases}
    {\rm diag}(1,a_1,\ldots, a_1\ldots a_n) & n \, \textnormal{even}\\
    2{\rm diag}(1,a_1,\ldots, a_1\ldots a_n) & n \, \textnormal{odd}.
\end{cases}\]
When $n$ is even we have
\[\sigma_i(1,a_1,\ldots, a_1\ldots a_n)=\sigma_i(a_1,\ldots, a_1 \ldots a_n)\]
as each monomial with $\lbrace 1\rbrace$ in it is equal to zero. For $n$ odd, first note that $\lbrace 2, 2 \rbrace = \lbrace 2, -1 \rbrace=0$ in $\K_2^{\bullet}(\bfk)$. Splitting off all the terms which include the first element, we have
\[\sigma_i(2,2a_1,\ldots, 2a_1\ldots a_n)=\lbrace 2 \rbrace\sigma_{i-1}(a_1,\ldots, a_1\ldots a_n)+\sigma_i(2a_1,\ldots, 2a_1\ldots a_n);\]
now in the second term note that an element in the form $\lbrace 2x_1, \ldots, 2x_i\rbrace$ can be split as
\[\lbrace 2x_1, \ldots, 2x_i\rbrace=\lbrace x_1, \ldots, x_i\rbrace + \lbrace 2 \rbrace\sum_{s \leq i}\prod_{q \neq s}\lbrace x_q \rbrace + \lbrace 2, 2 \rbrace (\ldots ).\]
Thus, as $\lbrace 2,2 \rbrace=0$, we have that
\begin{align*}
    \sigma_i(2,2a_1,\ldots, 2a_1\ldots a_n)
    &=\lbrace 2 \rbrace\sigma_{i-1}(a_1,\ldots, a_1\ldots a_n)
    \\&\phantom{==}+(2^n-i-1)\lbrace 2 \rbrace\sigma_{i-1}(a_1,\ldots, a_1\ldots a_n)
    \\&\phantom{==}+ \sigma_i(a_1,\ldots, a_1\ldots a_n)
    \\&=\sigma_i(a_1,\ldots, a_1\ldots a_n)+(2^n-i)\lbrace 2 \rbrace\sigma_{i-1}(a_1,\ldots, a_1\ldots a_n),
\end{align*}
where the $2^i-i-1$ coefficient corresponds to the number of ways each monomial in $\sigma_{i-1}$ can be obtained by splitting off a $2$ from a monomial in $\sigma_i$. This shows that 
\[\alpha_i^{\rm SW}(E_n)=\begin{cases}
 \sigma_i(a_1,\ldots, a_1\ldots a_n) & n \, \textnormal{even}, \\
 \sigma_i(a_1,\ldots, a_1\ldots a_n)+\lbrace 2 \rbrace \cdot \sigma_{i-1}(a_1,\ldots, a_1\ldots a_n) & n \,\textnormal{odd}, \, i \, \textnormal{even},\\
 \sigma_i(a_1,\ldots, a_1\ldots a_n) & n \,\textnormal{odd}, \, i \, \textnormal{odd}.
\end{cases}\]
Our objective is to prove the following:
\begin{prop}\label{prop:sigmastate}
We have 
\[
\sigma_i(a_1,\ldots, a_1\ldots a_n)=\begin{cases}
1 & i = 0, \\
0 & i \neq 0, 2^{n-1},\\
\varepsilon^{2^{n-1}-n}\cdot \lbrace a_1, \ldots, a_n \rbrace + (\ldots) & i = 2^{n-1}.
\end{cases}
\]
which implies
\[
\alpha_i(E_n)=\begin{cases}
1 & i = 0, \\
0 & i \neq 0, 2^{n-1},\\
\varepsilon^{2^{n-1}-n}\cdot \lbrace a_1, \ldots, a_n \rbrace + (\ldots) & i = 2^{n-1}.
\end{cases}
\]
\end{prop}

The second implication is a direct consequence of the first and the formula
\[
    \alpha_i(E)=\begin{cases}
 \alpha_i^{\rm SW}(E) &i \, \textnormal{even}, \\
 \alpha_i^{\rm SW}(E)+\lbrace 2 \rbrace \cdot \alpha_{i-1}(E) & i \, \textnormal{odd},
\end{cases}
\]
so we have to compute the classes $\sigma_i(a_1,a_2,\ldots,a_1a_2 \ldots a_n) \in \K_2^{\bullet}(F_n)$. To do so, note that there is a ring map (injective when $\bfk$ is totally real)
\[\FF_2\left[a_1,\ldots, a_n, \varepsilon\right]/I_n \to \K_2^{\bullet}(F_n)\]
where $I_n$ is the homogeneous ideal generated by the elements $a_j^2-\varepsilon a_j$, and the elements $\sigma_i(a_1,a_2,\ldots,a_1a_2 \ldots a_n)$ belong to the image of this map. More specifically, if we take the polynomial
\[p_n(a_1,\ldots,a_n)=\prod_{I \subset \lbrace 1, \ldots, n\rbrace}(1 + \sum_{j \in I} a_j),\]
then $\sigma_i(a_1,\ldots,a_1\ldots a_n)$ is the image of the component of degree $i$ of $p_n$.

We will compute specific coefficients of $\sigma_i$ by determining the corresponding coefficients of the polynomial $p_n$. This is achieved by analyzing the recursive relations satisfied by this family of polynomials and examining related auxiliary polynomials under the relations defined by the ideal~$I_n$.

Given indeterminate variables $x_0,x_1,\dots$, we defined the polynomial function 
\[
\Phi_n(x_0;x_1,\dots,x_n)\coloneqq\prod_{I\subset\{1,\dots,n\}}\left(x_0+\sum_{i\in I}x_i\right)\in \mathbb{Z}[x_0,x_1,\dots,x_n].
\]
Then, we have that the the aforementioned polynomial $p_n$ satisfies 
\[
p_n(x_1,\dots,x_n)= \Phi_n(1;x_1,\dots,x_n) \in \mathbb{Z}[x_1,\dots,x_n].
\]
Moreover, we define an auxiliary polynomial $q_n$ by
\[
q_n(x_1,\dots,x_n;x_{n+1})\coloneqq \Phi_n(1+x_{n+1};x_1,\dots,x_n) \in \mathbb{Z}[x_1,\dots,x_n,x_{n+1}].
\]
We compute properties of these polynomial by recursion. Let us remark that the polynomial function $\Phi_n$ satisfies the following recursive relation:
\[
    \Phi_n(x_0;x_1,\dots,x_n)=\Phi_{n-1}(x_0;x_1,\dots,x_{n-1})
    \cdot\Phi_{n-1}(x_0+x_n;x_1,\dots,x_{n-1}).
\]
\begin{table}[hb]
    \centering
    \begin{align*}
    p_0&=1 \\q_0(x_1)&=1+x_1\\
    p_1(x_1)&=1+x_1\\q_1(x_1;x_2)&= 1 + x_1 + 2x_2 + x_1 x_2 + x_2^2  \\  p_2(x_1,x_2)&=1+2x_1+2x_2+x_1^2+3x_1x_2+x_2^2+x_1^2x_2+x_1x_2^2\\
    q_2(x_1,x_2;x_3)&=1 + 2 x_2 + 2 x_1 +4 x_3 + x_1^2 + 3 x_1 x_2  + x_2^2 +6 x_3^2+ 6 x_1 x_3 
    \\&\phantom{=}+  6 x_2 x_3 + x_1^2 x_2  + x_1 x_2^2 + 6 x_1 x_2 x_3  + x_1^2 x_2 x_3 + 2 x_2^2 x_3  
    \\&\phantom{=}+ x_1 x_2^2 x_3 + 2 x_1^2 x_3 + 6 x_2 x_3^2 + 3 x_1 x_2 x_3^2 + 6 x_1 x_3^2  + x_2^2 x_3^2 + 
    \\&\phantom{=}x_1^2 x_3^2 + 4 x_3^3 + 2 x_2 x_3^3 + 2 x_1 x_3^3 +x_3^4 
    \end{align*}
    \caption{Examples of the initial values for the polynomials~$p_n$ and $q_n$.}
    \label{tab:exam}
\end{table}

First, we begin by establishing the following relation, which is satisfied by the polynomial function $\Phi$ modulo 2. In the following computations we will distinguish between the symbol $\equiv$ to denote mod $(2)$ equivalence and the symbol $=$ to denote equivalence with integral coefficients.
\begin{lm}  \label{lm:sum}
    The equality
    \[\Phi_{n}(y+z;x_1,\dots,x_n)\equiv \Phi_{n}(y;x_1,\dots,x_n)+\Phi_{n}(z;x_1,\dots,x_n)\]
    holds in $\mathbb{F}_2[x_1,\dots,x_n,y,z]$.
\end{lm}

\begin{proof}
    We proceed by induction on $n$. For $n=0$ the statement is trivial since $\Phi_0=\operatorname{Id}$. For $n=1$, we have that 
\begin{align*}
    \Phi_1(y+z;x_1)&=(y+z)(y+z+x_1)\\
    &=y^2+2yz+z^2+yx_1+zx_1\\
    &\equiv y(y+x_1)+z(z+x_1)\\
    &=\Phi_1(y;x_1)+\Phi(z;x_1).
\end{align*}
Let us assume the statement holds for integers less than $n$.
Writing $\Phi_{n-1}(x_0)$ for $\Phi_{n-1}(x_0;x_1,\dots,x_{n-1})$ we have that
\begin{align*}\displaystyle
    \Phi_n(y+z;x_1,\dots,x_n)
    &=\Phi_{n-1}(y+z)\cdot \Phi_{n-1}(y+z+x_n)
    \\&\equiv\left(\Phi_{n-1}(y)+\Phi_{n-1}(z)\right)
    \cdot\left(\Phi_{n-1}(y+x_n)+\Phi_{n-1}(z)\right)
    \\&=\Phi_{n-1}(y)\Phi_{n-1}(y+x_n)
    \\&\phantom{\equiv}+\Phi_{n-1}(z)
    \cdot\left(\Phi_{n-1}(y)+\Phi_{n-1}(y+x_n)+\Phi_{n-1}(z)\right)
    \\&\equiv\Phi_{n-1}(y)\Phi_{n-1}(y+x_n)+\Phi_{n-1}(z)\Phi_{n-1}(z+x_n)
    \\&=\Phi_{n}(y;x_1,\dots,x_n)+\Phi_{n}(z;x_1,\dots,x_n).
\end{align*}
\end{proof}

 Given a polynomial in several variables $f$, we denote by $\operatorname{pol}_i(f)$ the component of degree $i$ of $f$, i.e.\ the sum of summands of $f$ of degree $i$.

\begin{cor}\label{cr:coefpq}
    If $i<2^{n}$, then 
    \[
    \operatorname{pol}_i(p_n)\equiv \operatorname{pol}_i(q_n) \mod (2).
    \]
\end{cor}

\begin{proof}
    By \cref{lm:sum}, we have that
\begin{align*}
    q_n-p_n&=\Phi_n(1+x_{n+1};x_1,\dots,x_n)-\Phi_n(1;x_1,\dots,x_n)\\
    &\equiv\Phi_n(x_{n+1};x_1,\dots,x_n)\mod (2),
\end{align*}
which is a homogeneous polynomial of degree $2^n$.
\end{proof}

By definition, we have that the constant term $\operatorname{pol}_0(p_n)=1$. The following proposition shows that the monomials of degree $i<2^{n-1}$ vanish $\mod (2)$ and provides a recursive formula for the first non-trivial coefficients.

\begin{prop}\label{pr:relations}
    Let $i$ be an integer such that $0<i<2^{n-1}$. Then, the following relations hold in $\FF_2[x_1,\dots,x_{n+1}]$
\begin{enumerate}
  \item $\operatorname{pol}_i(p_n)\equiv 0$,
    \item $\operatorname{pol}_{2^{n-1}}(p_n)\equiv
(\operatorname{pol}_{2^{n-2}}(p_{n-1}))^2+
\operatorname{pol}_{2^{n-1}}(q_{n-1})$.
\end{enumerate}
\end{prop}

\begin{proof}
    We proceed by induction on $n$. For $n=0,1$ the statement is vacuous. In \cref{tab:exam}, we can witness that the  polynomials $p_2$ and $q_2$ satisfy the relations in the statement. Let us assume the statement hold for integers less than $n$.
    
    Subsequently, let us remark that the polynomials $p_n$ satisfy the recursion relation: 
\begin{align*}
\displaystyle
p_n(x_1,\dots,x_n)
    &=\prod_{I\subset\{1,\dots,n\}}\left(1+\sum_{i\in I}x_i\right)\\
    &=\prod_{I\subset\{1,\dots,n-1\}}\left(1+\sum_{i\in I}x_i\right)\cdot
\prod_{I\subset\{1,\dots,n-1\}}\left(1+x_n+\sum_{i\in I}x_i\right)\\
    &=p_{n-1}(x_1,\dots,x_{n-1})\cdot q_{n-1}(x_1,\dots,x_{n-1};x_n).
\end{align*}
      Then, for any $i\in\mathbb{N}$:
\begin{align*}
    \displaystyle\operatorname{pol}_i(p_n)
    &=\operatorname{pol}_i(p_{n-1}\cdot q_{n-1})\\
    &=\sum_{j=0}^i \operatorname{pol}_j(p_{n-1})\operatorname{pol}_{i-j} (q_{n-1})\\
    &=\sum_{j=0}^i \operatorname{pol}_j(p_{n-1})\operatorname{pol}_{i-j} (q_{n-1}).
\end{align*}
We have that $\operatorname{pol}_j(p_{n-1}),\operatorname{pol}_{j-i}(q_{n-1})\equiv0\mod(2)$ for $0<i,j-i<2^{n-2}$ by induction hypothesis and by by \cref{cr:coefpq}. Thus, if~$0<i<2^{n-1}$, then the only non-vanishing couple of indices are $i,j\in\{0,2^{n-2}\}$. Hence, we have that $\operatorname{pol}_i(p_n)\equiv0\mod(2)$ for $i\neq2^{n-2}$, and that
\[
    \displaystyle\operatorname{pol}_{2^{n-2}}(p_n)
    \equiv \operatorname{pol}_{2^{n-2}}(p_{n-1})+\operatorname{pol}_{2^{n-2}} (q_{n-1}) \mod(2).
\]
However, by \cref{cr:coefpq}, 
$\operatorname{pol}_{2^{n-2}}(p_{n-1})\equiv\operatorname{pol}_{2^{n-2}} (q_{n-1}) \mod (2).$ Which yields the first relation. Lastly, since $\operatorname{pol}_0(p_{n-1})=\operatorname{pol}_0(q_{n-1})=1$, then
\begin{align*}
    \operatorname{pol}_{2^{n-1}}(p_n)&
    = \operatorname{pol}_{2^{n-1}}(p_{n-1})+
    \operatorname{pol}_{2^{n-2}}(p_{n-1})\operatorname{pol}_{2^{n-2}} (q_{n-1})
    +\operatorname{pol}_{2^{n-1}} (q_{n-1})
    \\&\equiv [\operatorname{pol}_{2^{n-2}}(p_{n-1})]^2
    +\operatorname{pol}_{2^{n-1}} (q_{n-1}) \mod (2),
\end{align*}
since $p_{n-1}$ has degree $2^{n-1}-1$ and 
$\operatorname{pol}_{2^{n-2}}(p_{n-1})\equiv\operatorname{pol}_{2^{n-2}} (q_{n-1})$ by \cref{cr:coefpq}.
\end{proof}

In order to utilize the second relation in \cref{pr:relations} to compute the coefficients of $p_n$, we first determine the polynomial part of the highest degree of $q_n$:
\begin{align}
    \operatorname{pol}_{2^{n}}(q_n)
    &=\operatorname{pol}_{2^{n}}\left(\prod_{I\subset\{1,\dots,n\}}\left(1+x_{n+1}+\sum_{i\in I}x_i\right)\right)\label{eq:coefqn}\\ \nonumber
    &=\prod_{I\subset\{1,\dots,n\}}\operatorname{pol}_{1}\left(1+x_{n+1}+\sum_{i\in I}x_i\right)\\ \nonumber
    &=\prod_{I\subset\{1,\dots,n\}}\left(x_{n+1}+\sum_{i\in I}x_i\right)\\ \nonumber
    &=\Phi_n\left(x_{n+1};x_1,\dots,x_n\right).
\end{align}
Subsequently, this formulation is applied to the relation provided in \cref{pr:relations}. By iterating the process, we derive the following formula:

\begin{align}\label{eq:pnrel}
    \operatorname{pol}_{2^{n-1}}(p_n)&\equiv
    (\operatorname{pol}_{2^{n-2}}(p_{n-1}))^2+\operatorname{pol}_{2^{n-1}}(q_{n-1})\\\nonumber
    &\equiv
    \left((\operatorname{pol}_{2^{n-3}}(p_{n-2}))^2+\operatorname{pol}_{2^{n-2}}(q_{n-2}) \right)^2
    +\operatorname{pol}_{2^{n-1}}(q_{n-1})\\\nonumber
    &\equiv
    (\operatorname{pol}_{2^{n-3}}(p_{n-2}))^4+(\operatorname{pol}_{2^{n-2}}(q_{n-2}))^2
    +\operatorname{pol}_{2^{n-1}}(q_{n-1})
    \\\nonumber
    &\equiv \displaystyle\sum_{i=1}^n (\operatorname{pol}_{2^{n-i}}(q_{n-i}))^{2^{i-1}}\\\nonumber
    &\equiv \displaystyle\sum_{i=0}^{n-1} (\operatorname{pol}_{2^{i}}(q_{i}))^{2^{(n-1)-i}}\\\nonumber
    &\mkern-4mu\overset{\eqref{eq:coefqn}}{\equiv} \displaystyle\sum_{i=0}^{n-1} \Phi_i(x_{i+1};x_1,\dots,x_i)^{2^{(n-1)-i}}.
\end{align}
To compute the image of the coefficient in $\mathbb{F}_2[x_1\dots,x_n,\varepsilon]/I_n$, where as above, the ideal $I_n$ is generated by $x_i^2-\varepsilon x_i$, for $i=1,\dots,n$.
We evaluate the image of each summand individually:
\begin{align}\label{eq:PhiIn}
    &\Phi_i(x_{i+1};x_1,\dots,x_i)^{2^{(n-1)-i}}\\\nonumber
    & \equiv \Phi_i(x_{i+1}^{2^{(n-1)-i}};x_1^{2^{(n-1)-i}},\dots,x_i^{2^{(n-1)-i}})\\\nonumber
    & \equiv \Phi_i(\varepsilon^{2^{(n-1)-i}-1} x_{i+1};\varepsilon^{2^{(n-1)-i}-1} x_1,\dots,\varepsilon^{2^{(n-1)-i}-1} x_i)\\\nonumber
    & \equiv \left(\varepsilon^{2^{(n-1)-i}-1} \right)^{2^i} \Phi_i( x_{i+1};x_1,\dots,x_i)\\\nonumber
    & \equiv \varepsilon^{2^{(n-1)}-2^i} \Phi_i( x_{i+1};x_1,\dots,x_i).
\end{align}
Similarly, we compute in advance the image of the factors under the relations defined by the ideal $I_n$. Let $I\subset \{1,\dots,n-1\}$, then
\begin{align}
    &\left(x_n+x_{n+1}+\sum_{i\in I}x_i\right)
    \left(x_{n+1}+\sum_{i\in I}x_i\right) \label{eq:twofacts}\\ \nonumber
    &=x_n\left(x_{n+1}+\sum_{i\in I}x_i\right)
    +\left(x_{n+1}+\sum_{i\in I}x_i\right)^2\\\nonumber
    &\equiv x_n\left(x_{n+1}+\sum_{i\in I}x_i\right)
    +x_{n+1}^2+\sum_{i\in I}x_i^2\\ \nonumber
    &\equiv x_n\left(x_{n+1}+\sum_{i\in I}x_i\right)
    +\varepsilon x_{n+1}+\varepsilon\sum_{i\in I}x_i\\ \nonumber
    &\equiv (x_n+\varepsilon)\left(x_{n+1}+\sum_{i\in I}x_i\right).
\end{align}
Applying this formula to the recursive relation of the function $\Phi$ results in the following expression:

\begin{align}\label{eq:Phirel}
    &\Phi_n(x_{n+1};x_1,\dots,x_n)\\\nonumber
    &=\Phi_{n-1}(x_n+x_{n+1};x_1,\dots,x_{n-1})\cdot \Phi_{n-1}(x_{n+1};x_1,\dots,x_{n-1})\\\nonumber
    &=\prod_{I\subset\{1,\dots,n-1\}}\left(x_n+x_{n+1}+\sum_{i\in I}x_i\right)
    \left(x_{n+1}+\sum_{i\in I}x_i\right)\\\nonumber
    &\mkern-4mu\overset{\eqref{eq:twofacts}}{\equiv} \prod_{I\subset\{1,\dots,n-1\}} (x_n+\varepsilon)\left(x_{n+1}+\sum_{i\in I}x_i\right)\\\nonumber
    &=(x_n+\varepsilon)^{2^{n-1}}\prod_{I\subset\{1,\dots,n-1\}} \left(x_{n+1}+\sum_{i\in I}x_i\right)\\\nonumber
    &=(x_n+\varepsilon)^{2^{n-1}}\Phi_{n-1}(x_{n+1};x_1,\dots,x_{n-1})\\\nonumber
    &=\prod_{i=1}^{n} (x_i+\varepsilon)^{2^{i-1}}\cdot\Phi_{0}(x_{n+1})\\\nonumber
    &=\prod_{i=1}^{n} (x_i+\varepsilon)^{2^{i-1}}\cdot(x_{n+1}).
\end{align}
We compute the image of the first factor under the relations defined by the ideal~$I_n$:
\begin{align}\label{eq:xipluse}
    \displaystyle \prod_{i=1}^{n} (x_i+\varepsilon)^{2^{i-1}}&\equiv \prod_{i=1}^{n} (x_i^{2^{i-1}}+\varepsilon^{2^{i-1}})\\\nonumber
    &\equiv \prod_{i=1}^{n} (\varepsilon^{2^{i-1}-1} x_i+\varepsilon^{2^{i-1}})\\\nonumber
    &\equiv \prod_{i=1}^{n} \varepsilon^{2^{i-1}-1}( x_i+\varepsilon)\\\nonumber
    &= \varepsilon^{2^{n}-1-n} \prod_{i=1}^{n} ( x_i+\varepsilon).
\end{align}
Lastly, we have that the polynomial part of degree $2^{n-1}$ of $p_n$ is then given by:
\begin{align}\label{eq:coeffofpn}
    \operatorname{pol}_{2^{n-1}}(p_n)
    &\mkern-4mu\overset{\eqref{eq:pnrel}}{\equiv} \displaystyle\sum_{i=0}^{n-1} \Phi_i(x_{i+1};x_1,\dots,x_i)^{2^{(n-1)-i}}\\\nonumber
    &\mkern-4mu\overset{\eqref{eq:PhiIn}}{\equiv} \displaystyle\sum_{i=0}^{n-1} \varepsilon^{2^{(n-1)}-2^i} \Phi_i( x_{i+1};x_1,\dots,x_i)\\\nonumber
    &\mkern-4mu\overset{\eqref{eq:Phirel}}{\equiv} \displaystyle\sum_{i=0}^{n-1} \varepsilon^{2^{(n-1)}-2^i} \prod_{j=1}^{i} (x_j+\varepsilon)^{2^{j-1}}\cdot(x_{i+1})\\\nonumber
    &\mkern-4mu\overset{\eqref{eq:xipluse}}{\equiv} \displaystyle\sum_{i=0}^{n-1} \varepsilon^{2^{(n-1)}-2^i} \varepsilon^{2^{i}-1-i} \prod_{j=1}^{i} (x_j+\varepsilon)\cdot(x_{i+1})\\\nonumber
    &\equiv \displaystyle\sum_{i=0}^{n-1} \varepsilon^{2^{(n-1)}-1-i} \prod_{j=1}^{i} (x_j+\varepsilon)\cdot(x_{i+1})\\\nonumber
    &\equiv \displaystyle\sum_{i=1}^{n} \varepsilon^{2^{(n-1)}-i} s_{i}({x_1,\dots,x_n}),
\end{align}
where $s_i$ is the $i$-th elementary symmetric polynomial in $x_1,\dots,x_n$ seen as element in $\mathbb{F}_2[x_1\dots,x_n,\varepsilon]/I_n$. This yields \cref{prop:sigmastate}.

\subsection{Cohomological invariants from theta characteristics}

We are ready to compute the Galois-Stiefel-Whitney classes of theta characteristics on our test curve.

\begin{prop}
We have 

\[
\alpha_i(S^{-}_{C_g})=\begin{cases}
    1        &i=0, \\
    0        &0<i< 2^{g-2}, \\
\varepsilon^{2^{g-2}-g+1}(\sum_{i=1}^g \prod_{j \neq i} \lbrace a_j \rbrace) + * \quad                  &i=2^{g-2},
\end{cases} 
\]
and
\[
\alpha_i(S_{C_g})=\begin{cases}
1 & i=0, \\
0 & 0<i< 2^{g-1}, \\
\varepsilon^{2^{g-1}-g}\prod_{i=1}^g \lbrace a_j \rbrace + * & i=2^{g-1},
\end{cases}
\]
where $\varepsilon$ appears at a strictly higher power in the $*$ summands than on the summands on the left.

\end{prop}
\begin{proof}
We begin with the odd theta characteristics. We have
\[
\alpha_{\rm tot}(S^{-}_{C_g}) = \prod_{I \subseteq \lbrace 1, \ldots, g \rbrace} \alpha_{\rm tot}(E_I)^{n^{-}_I}=\prod_{I \subseteq \lbrace 1, \ldots, g \rbrace, \mid I \mid < g} \alpha_{\rm tot}(E_I)^{2^{g-\mid I \mid}}.
\]
Now note that 
\[\alpha_{\rm tot}(E_I)^{2^{g-\mid I \mid}-1} = (1+\alpha_{2^{\mid I \mid - 1}}(E_{I}))^{2^{g-\mid I \mid}-1} = 1+\alpha_{2^{\mid I \mid - 1}}(E_I)^{2^{g-\mid I \mid -1}}.\]
which shows that the first nontrivial term has degree $2^{\mid I \mid - 1}\cdot 2^{g-\mid I \mid -1} = 2^{g-2}$, which in turn shows that the term of degree $2^{g-2}$ of $\alpha_{\rm tot}(S_{C_g}^-)$ is
\[\alpha_{\rm tot}(S^{-}_{C_g})=\sum_{I \subseteq \lbrace 1, \ldots, g \rbrace, \mid I \mid < g} \alpha_{2^{\mid I \mid - 1}}(E_I)^{2^{g-\mid I \mid-1}} = \sum_{I \subseteq \lbrace 1, \ldots, g \rbrace, \mid I \mid < g} \varepsilon^{2^{g}-2-\mid I \mid}\prod_{i \in I}\lbrace a_i \rbrace .\]
The formula for all characteristic classes follows immediately from the fact that
\[\alpha_{\rm tot}(S^+_{C_g})=\alpha_{\rm tot}(S^-_{C_g})\alpha_{\rm tot}(E_{\lbrace 1, \ldots, g \rbrace}),\]
which shows that 
\[\alpha_{\rm tot}(S_{C_g})=\alpha_{\rm tot}(S^-_{C_g})\cdot \alpha_{\rm tot}(S^+_{C_g})=\alpha_{\rm tot}(S^-_{C_g})^2\alpha_{\rm tot}(E_{\lbrace 1, \ldots, g \rbrace}).\]
\end{proof}

We can finally gather our results together to prove Thm. \ref{thm:classes}, and in fact a bit more. Recall that as ${\rm Gal}(\mathbb{C}/\mathbb{R})=\ZZ/2\ZZ$, we have
\[{\rm H}^{\bullet}(\mathbb{R},\ZZ/2\ZZ)={\rm H}^{\bullet}_{\rm grp}(\ZZ/2\ZZ,\ZZ/2\ZZ)\simeq\FF_2\left[ t \right].\]
Thus, since ${\rm H}^1(\mathbb{R},\ZZ/2\ZZ)=\mathbb{R}^*/(\mathbb{R}^*)^2$ is generated by $\varepsilon = \lbrace -1\rbrace$, we conclude that 
\[\K_2^{\bullet}(\mathbb{R})=\H_{\ZZ/2\ZZ}^{\bullet}(\mathbb{R})=\FF_2\left[ \varepsilon \right],\]
an integral domain.

\begin{thm}\label{thm: main proof B}
Let $\bfk=\mathbb{R}$ be the field of real numbers, and let $\Mcal_g, \Mcal_g^{\rm ct}$ and $\mathcal{A}_g$ be respectively the moduli stacks of smooth genus $g$ curves, compact type genus $g$ curves and principally polarized abelian varieties of dimension $g$ over $\mathbb{R}$.

Let $\alpha_i(\mathcal{S}^-_g), \alpha_i(\mathcal{S}^+_g), \alpha_i(\mathcal{S}_g)$ be the Galois-Stiefel-Whitney classes induced respectively by odd, even and all theta characteristics. These belong to ${\rm Inv}^i(\Mcal_g^{\rm ct}, \K_2)$ (and to $\H^i(\Mcal_g^{\rm ct},\ZZ/2\ZZ)$), and by restriction to ${\rm Inv}^i(\Mcal_g)$ (resp. $\H^i(\Mcal_g,\ZZ/2\ZZ)$) and ${\rm Inv}^i(\Hcal_g,\K_2)$ (resp. $\H^i(\Hcal_g,\ZZ/2\ZZ)$). Finally let $\alpha_i((\Xcal_g)_2)$ be the Galois Stiefel Whitney class defined by the two-torsion subgroup of the universal abelian variety in ${\rm Inv}^i(\mathcal{A}_g, \K_2)$ and $\H^i(\mathcal{A}_g,\ZZ/2\ZZ)$. Then:

\begin{enumerate}
    \item The classes $1, \alpha_{2^{g-2}}(\mathcal{S}^-_g)$ and $\alpha_{2^{g-1}}(\mathcal{S}_g)$ are $\K_2^{\bullet}(\mathbb{R})$-linearly independent.
    \item The classes $1, \alpha_{\rm tot}(\mathcal{S}^-_g)$ and $\alpha_{\rm tot}(\mathcal{S}_g)$ are $\K_2^{\bullet}(\mathbb{R})$-linearly independent.
    \item The classes $1, \alpha_{\rm tot}(\mathcal{S}^-_g)$ and $\alpha_{\rm tot}(\mathcal{S}^+_g)$ are $\K_2^{\bullet}(\mathbb{R})$-linearly independent.
    \item The classes $1$ and $\alpha_{2^{g-1}}((\Xcal_g)_2)$ are $\K_2^{\bullet}(\mathbb{R})$-linearly independent.
    \item In particular, all the classes above, seen as classes in \'etale cohomology\footnote{Here we see $\alpha_{\rm tot}$ as a finite sum as the classes $\alpha_i$ will eventually be zero as cohomological invariants.}, are nontrivial and do not belong to the $\H^{\bullet}(\mathbb{R},\ZZ/2\ZZ)$-submodule generated by the image of the cycle map.
\end{enumerate}
\end{thm}
\begin{proof}
To prove point (1), assume we have an equality
\[x_0 + x_1 \cdot \alpha_{2^{g-2}}(\mathcal{S}^-_g) + x_2 \cdot \alpha_{2^{g-1}}(\mathcal{S}_g) = 0 \in \Inv(\Mcal_g^{\rm ct}, \K_2),\]
with $x_0, x_1, x_2 \in \K_2^\bullet(\mathbb{R})$. Pulling it back trough the morphism $C_g:\Spec(F) \to \Mcal_g^{\rm ct}$ we obtain the equality
\[x_0 + x_1 \cdot \alpha_{2^{g-2}}(S^-_{C_g}) + x_2 \cdot \alpha_{2^{g-1}}(S_{C_g}) = 0 \in \K_2^{\bullet}(F).\]
Now we apply the sequence of ramification maps ${\rm P}=\partial_{a_g=0}\circ \ldots \circ\partial_{a_1=0}$ to the equation. Clearly 
\[{\rm P}(x_0)={\rm P}(x_1 \cdot \alpha_{2^{g-2}}(S^-_{C_g}))=x_1 \cdot P(\alpha_{2^{g-2}}(S^-_{C_g}))=0.\]
As the each component of $\alpha_{2^{g-2}}(S^-_{C_g}))$ contains at most $g-1$ elements out of $a_1,\ldots, a_g$. On the other hand
\[{\rm P}(x_2 \cdot \alpha_{2^{g-1}}(\mathcal{S}_g))=x_2 \cdot {\rm P}(\alpha_{2^{g-1}}(\mathcal{S}_g))=x_2 \cdot \varepsilon^{2^{g-1}-g}.\]
As multiplying by $\varepsilon$ is injective in $\K_2(\mathbb{R})$ we conclude that $x_2$ must be zero. This also immediately implies that $x_0,x_1$ are zero as $\alpha_{2^{g-2}}(\mathcal{S}^-_g)$ does not belong to $\K_2^{\bullet}(\mathbb{R})$.

Point (2) is a direct consequence of point (1) as if we had an equality
\[y_0 + y_1 \cdot \alpha_{\rm tot}(\mathcal{S}^-_g) + y_2 \cdot \alpha_{\rm tot}(\mathcal{S}_g) = 0 \in \Inv(\Mcal_g^{\rm ct}, \K_2),\]
then restricting the elements to $C_g$ and taking the lowest degree components we would get again a relation between some multiple of $1$, $\alpha_{2^{g-2}}(S^-_{C_g}))$ and $\alpha_{2^{g-1}}(S_{C_g}))$.

To prove point (3), assume that
\[x_0 + x_1 \cdot \alpha_{\rm tot}(\mathcal{S}^-_g) + x_2 \cdot \alpha_{\rm tot}(\mathcal{S}^+_g) = 0 \in \Inv(\Mcal_g^{\rm ct}, \K_2).\]
Then as $\alpha_{\rm tot}(\mathcal{S}_g)=\alpha_{\rm tot}(\mathcal{S}^-_g)\alpha_{\rm tot}(\mathcal{S}^+_g)$ we should have
\[x_2 \cdot \alpha_{\rm tot}(\mathcal{S}_g)=\alpha_{\rm tot}(\mathcal{S}^-_g)(x_0 + x_1\cdot \alpha_{\rm tot}(\mathcal{S}^-_g))=x_0\cdot \alpha_{\rm tot}(\mathcal{S}_g^{-})+ x_1 \cdot \alpha_{\rm tot}(\mathcal{S}_g^{-})^2.\]
Now write $x_i'$ for the lowest degree component of $x_i$: the left hand side's lowest degree term is $x'_2 \cdot \alpha_{2^{g-1}}(\mathcal{S}_g)$, while the lowest degree term of the right hand side is of the form 
\[x'_0 \cdot \alpha_{2^{g-2}}(S^-_{C_g})+x'_1 \cdot \alpha_{2^{g-2}}(S^-_{C_g})^2=(x'_0+x'_1\cdot \varepsilon^{2^{g-2}})\alpha_{2^{g-2}}(S^-_{C_g})\]
and by point (1) these cannot be equal.

Point (3) is immediate as ${\rm Jac}(C_g)_2 = S_{C_g}$.

Point (4) follows from Prop. \ref{prop: negligible}, as point (1) shows that each of these classes induces a nontrivial cohomological invariant.
\end{proof}

When $\bfk$ is a totally real field (or a purely transcendental extension of one) Thm.~\ref{thm: main proof B} gives us the following obvious corollary:

\begin{cor}
Let $\bfk$ be a totally real field of a purely transcendental extension of one. Then
\begin{enumerate}
    \item The $\K_2^{\bullet}(\bfk)$-submodule of $\Inv(\Mcal_g^{\rm ct},\K_2)$ generated by the elements 
    \[1, \alpha_{2^{g-2}}(\mathcal{S}^-_g), \text{ and }\alpha_{2^{g-1}}(\mathcal{S}_g)\] has rank $3$.
    \item The $\K_2^{\bullet}(\bfk)$-submodule of $\Inv(\Mcal_g^{\rm ct},\K_2)$ generated by the elements  
    \[1, \alpha_{\rm tot}(\mathcal{S}^-_g), \text{ and } \alpha_{\rm tot}(\mathcal{S}_g)\] has rank $3$.\item The $\K_2^{\bullet}(\bfk)$-submodule of $\Inv(\Mcal_g^{\rm ct},\K_2)$ generated by the elements  
    \[1, \alpha_{\rm tot}(\mathcal{S}^-_g), \text{ and } \alpha_{\rm tot}(\mathcal{S}^+_g)\] has rank $3$.
    \item The $\K_2^{\bullet}(\bfk)$-submodule of $\Inv(\mathcal{A}_g,\K_2)$ generated by $1$ and $\alpha_{2^{g-1}}((\Xcal_g)_2)$ has rank $2$.
    \item In particular the classes above, seen as classes in \'etale cohomology, are nontrivial and do not belong to the $\H^{\bullet}(\bfk,\ZZ/2\ZZ)$-submodule generated by the image of the cycle map.
\end{enumerate}
\end{cor}

\end{document}